\documentclass[11pt]{amsart}
\usepackage{amsmath}
\usepackage{amssymb}
\usepackage{float}   
\usepackage[normalem]{ulem}

\usepackage{hyperref}
\usepackage{graphicx}
\usepackage[all]{xy}
\usepackage{latexsym}
\usepackage{tikz}
\usepackage{enumerate}

\usetikzlibrary{matrix,arrows,backgrounds,shapes.misc,shapes.geometric,patterns,calc,positioning}
\usetikzlibrary{decorations.pathmorphing}

\theoremstyle{theorem}
\newtheorem{theorem}{Theorem}[section]

\newtheorem{lemma}[theorem]{Lemma}
\newtheorem{cor}[theorem]{Corollary}
   
\newtheorem{prop}[theorem]{Proposition}

\theoremstyle{remark}
\newtheorem{rem}[theorem]{Remark}

\newtheorem{ex}[theorem]{Example}

\theoremstyle{definition}
\newtheorem{defi}[theorem]{Definition} 

\newcommand{\mods}{\operatorname{mod}\nolimits}
\newcommand{\Ext}{\operatorname{Ext}\nolimits}
\newcommand{\Hom}{\operatorname{Hom}\nolimits}
\newcommand{\tp}{\operatorname{top}\nolimits}
\newcommand{\soc}{\operatorname{soc}\nolimits}
\newcommand{\pdim}{\operatorname{pdim}\nolimits}
\newcommand{\idim}{\operatorname{idim}\nolimits}
\newcommand{\coker}{\operatorname{coker}\nolimits}
\newcommand{\C}{\mathcal C}

\newcommand{\LT}{\Lambda(T)}

\usepackage{color}

\title{Higher extensions for gentle algebras}
\author{Karin Baur, Sibylle Schroll}

\address{Institute of Mathematics and Scientific Computing, NAWI Graz, 
University of Graz, Heinrichstrasse 36, 8010 Graz, Austria. 
School of Mathematics,
University of Leeds, LS2 9JT, Leeds, UK  }
\email{ka.baur@me.com}

\address{Department of Mathematics, University of Leicester, Leicester LE1 7RH, UK }
\email{schroll@leicester.ac.uk}


\thanks{The first author was supported by a Royal Society Wolfson Fellowship and by 
the Austrian Science Fund Project Number P30549 and DK W1230. This work has also been supported through an EPSRC 
Early Career Fellowship EP/P016294/1 held by  the second author. 
}


\begin{document}

\maketitle

\begin{abstract}
In this paper we determine extensions of higher degree between indecomposable 
modules over gentle algebras. In particular, our results show how such
extensions either eventually vanish or 
become periodic. 
We give a geometric interpretation of vanishing and periodicity of higher extensions in 
terms of the surface underlying the gentle algebra. For 
gentle algebras arising from triangulations of surfaces, we give an explicit basis of 
higher extension spaces between 
indecomposable modules. 
\end{abstract}

\section*{Introduction}
Extensions and higher extensions play an important role in many 
areas of mathematics.  
For example, in representation theory associated to cluster algebras, vanishing of 
extensions define cluster tilting subcategories. 
The vanishing of higher extensions give rise to cotorsion pairs and to Ext-orthogonal pairs of 
subcategories \cite{KS2010}. Extensions in the case of gentle algebras have 
been considered
by various authors, see, for example, \cite{Zhang14, CS17, CPS17, BDMTY17}. 
However, so far, higher extensions of gentle algebras have not been explicitly studied.

Geometric models as in the context of cluster theory \cite{ABCP, FST} play an important role in relating different areas of mathematics and transferring ideas from one area to another. This is demonstrated by recent advances in relating representation theory with homological mirror symmetry, see for example  \cite{B16, HKK17, Qiu15}.  In this context, in 
\cite{OPS} the authors give a geometric model of the bounded derived category of a 
finite dimensional gentle algebra $\Lambda$. 
In the case that the gentle algebra is homologically smooth (and with grading 
concentrated in degree 0),  this model coincides with the model in \cite{HKK17} shown to 
be equivalent to the partially wrapped Fukaya category of the same surface with 
stops \cite{LP18}. 
The geometric models in \cite{HKK17, LP18, OPS} 
give the higher extensions between indecomposable objects, in that 
$\Ext_{\Lambda}^n(M,N) \cong \Hom_{D^b(\Lambda)} (M,N[n])$. However, since in this 
set-up, the identification of 
$D^b(\Lambda)$ with the homotopy category of 
projectives $K^{-, b}({\rm proj} \Lambda)$ is used, 
it is difficult in practice to concretely deduce higher extensions between indecomposable 
modules over gentle algebras in terms of the surface geometry. For this reason we work instead with the geometric surface model of the module category~\cite{BCS2018} where arcs correspond directly to indecomposable modules.

In this paper, we determine the higher extension spaces between indecomposable 
modules over gentle algebras by a 
detailed study of their syzygies. We note that syzygies of indecomposable modules 
over string algebras have been studied in \cite{HZS05} and over gentle algebras 
 in~\cite{kalck}. However, for our purposes, we need a more precise understanding 
of the syzygies. 
We show that the first syzygy of any indecomposable module is a direct sum 
of indecomposable 
projectives and at most two non-projective uniserial modules which are generated by 
an arrow 
(Corollary~\ref{cor:ext2}). 
We use this to show that 
that higher extensions over gentle algebras 
are either eventually zero or become periodic (Theorem~\ref{thm:characterisation}).

We geometrically determine higher extensions for Jacobian algebras 
of quivers with 
potential arising from triangulations of oriented surfaces with marked points in the boundary 
(Theorems~\ref{thm:dim-ext2}, \ref{thm:ext3}, \ref{thm:ext4}), 
using the geometric model developed in 
\cite{ ABCP, FST, Labardini}. 
We also briefly geometrically describe higher extensions for general gentle algebras using 
the geometric model of the module category of a gentle algebra given in  \cite{BCS2018}, 
building on \cite{ABCP, CSP, DRS,Schroll15}.

In particular, we note that if two indecomposable modules are given by arcs in a surface, 
the number of crossings of the two arcs has no bearing on the higher extensions. 
Instead, the higher extensions solely depend on where the two arcs start and end.


\subsection*{Acknowledgement}


The authors thank I. Canakci, R. Coelho Simoes, and A. Garcia Elsener 
for helpful discussions. 
The authors also thank A. Garcia Elsener for comments  
on an earlier version of the paper.


\section{Notation}


Let $\Lambda = KQ/I$ be a finite dimensional algebra for some algebraically closed field $K$. 
We consider the category $\mods-\Lambda$ of 
finitely generated right modules over $\Lambda$. 
We write $p=a_1\cdots a_n$ for the composition of arrows $\stackrel{a_1}{\to}\dots\stackrel{a_n}{\to}$ 
and if $\alpha$ is an arrow from vertex $i$ to vertex $j$, we write $s(\alpha) = i$ and $t(\alpha) = j$. 
For $M\in \mods$-$\Lambda$, write minimal projective resolutions as 
\[
\dots \to P^2 \to P^1 \to P^0 \to M\to 0 
\]
Denote by  $P(M)$  the projective cover of $M$. If $P(M)\stackrel{f}{\to} M$ 
is the projective cover of $M$ then 
the first syzygy of $M$ is defined as $\Omega^1(M)=\ker f$, and for $k>1$, 
$\Omega^k(M)=\Omega^1(\Omega^{k-1}(M))$. 
Similarly, we write $I(M)$ for the injective hull of $M$ and if $ M \stackrel{g}{\to} I(M)$ is the injective hull of $M$ then the first co-syzygy of $M$ is define 
as $\Omega^{-1} (M) = \coker g$, and for $k > 1$, $\Omega^{-k} (M) =  \Omega^{-1} (\Omega^{-k+1}(M))$. 

%
\subsection{Gentle Algebras}

\begin{defi}
We call a finite dimensional algebra $\Lambda = KQ/I$ with admissible ideal  $I$ \emph{gentle} if 
\begin{enumerate}
\item For every $v \in Q_0$ there are at most two arrows starting and two arrows ending at $v$.
\item For every $\alpha \in Q_1$ there exists at most one arrow $\beta$ such that $\alpha \beta \notin I$ and there exists at most one arrow $\gamma$ such that $\gamma \alpha \notin I$.
\item For every $\alpha \in Q_1$ there exists at most one arrow $\beta$ such that $t(\alpha) = s(\beta)$ and $\alpha \beta \in I$ 
and there exists at most one arrow $\gamma$ such that $t(\gamma) = s (\alpha)$ and $\gamma \alpha \in I$.
\item $I$ is generated by paths of length two.
\end{enumerate} 
\end{defi}

%
\subsection*{Notation for string modules}

For $\Lambda=KQ/I$ gentle, the indecomposable modules are string and band modules, 
\cite{WW, BR}. String modules are given by walks in the quiver, as we recall now. 
Write 
$Q_1^{-1}$ for the formal inverses of the elements of $Q_1$. 
For $\alpha\in Q_1$, set 
$s(\alpha^{-1})=t(\alpha)$ and $t(\alpha^{-1})=s(\alpha)$.

A {\em walk} is a word $w=\alpha_1\cdots \alpha_m$ with $\alpha_i\in Q_1\cup Q_1^{-1}$ such 
that $t(\alpha_i)=s(\alpha_{i+1})$ for all $i$. 
A walk is a {\em string of length $m\ge 0$} if for all $i$, $\alpha_{i+1}\ne \alpha_i^{-1}$ and 
$\alpha_i\alpha_{i+1}\notin I$ if $\alpha_i,\alpha_{i+1}\in Q_1$ and 
$\alpha_{i+1}^{-1}\alpha_i^{-1}\notin I$ if $\alpha_i,\alpha_{i+1}\in Q_1^{-1}$. 
A string of length $0$ is a trivial string at some vertex $v$, it will be denoted by $e_v$. 
A string $w=\alpha_1\cdots \alpha_m$ is {\em direct} if all $\alpha_i$ occuring are in $Q_1$, 
it is {\em inverse} if all $\alpha_i$ occuring are in $Q_1^{-1}$.

We denote by $M(w)$ the string module associated to the string $w$. 
If $w= e_v$ is a trivial string at vertex $v$ then $M(e_v)$ is the simple module at $v$ which we also denote  by $S(v)$. 

We note that band modules are indexed by 1-parameter families. The  modules at the mouth of rank one tubes in the 
Auslander-Reiten quiver are called {\em quasi-simple} and up to cyclic permutation and inverse, they are given by cyclic walks in $Q$. 

\begin{rem}\label{rem:proj-ind}
Note that the indecomposable projective  $\Lambda$-modules are of the form 
$ P_{v} = e_v \Lambda = M(p^{-1}e_vq)$ where $p$ and $q$ are maximal paths starting from $v$. The indecomposable injective $\Lambda$-modules are of the form $ I_v = D(\Lambda e_v) = M(p e_v q^{-1})$  where $p$ and $q$ are maximal paths ending at the vertex $v$.
\end{rem}

%
\subsection{Gentle algebras arising from triangulations}\label{sec:triangul-algebra}

\ 

\noindent
Let  $T$ be an (ideal) triangulation of a surface with marked points  $(S,\mathcal M)$ 
such that  
$\mathcal M\subset \partial S\ne \emptyset$. 
By \cite{ABCP, Labardini} the triangulation $T$ defines an algebra $\Lambda=\LT$ 
given by a quiver $Q=Q(T)$, 
$Q=(Q_0,Q_1)$ with vertices $Q_0$ and arrows $Q_1$, and 
maps $s,t:Q_1\to Q_0$ sending an arrow to its starting or terminating point 
respectively and an ideal $I$ induced by the natural potential (arising from the triangulation of the surface) 
and define $\Lambda(T)=KQ/I$. By \cite{ABCP}, $\LT$ is a gentle algebra. 
Recall that in particular, any oriented cycle in $Q$ is a 3-cycle with full relations.  
\begin{defi}\label{def:surface-alg}
We call an algebra $\Lambda(T)=KQ/I$ as above a {\em the gentle algebra of the triangulation $T$}. 
\end{defi}

An \emph{arc} is a homotopy class of curves connecting marked points on the boundary. By \cite{ABCP} arcs are in bijection with 
isomorphism classes indecomposable string modules over $\Lambda(T)$ and the one parameter families of band modules are in 
bijection with homotopy classes of closed curves in $(S, M)$.

Any gentle algebra arising from a triangulation is 1-Gorenstein. 
For a general gentle algebra $\Lambda$, by \cite{kalck} $M \in \mods$-$\Lambda$ is Gorenstein projective if and only if 
$M = \Omega^d(N)$ for some module $N$ where $d$ is the virtual/Gorenstein 
dimension of $\Lambda$. That is for a gentle algebra $\Lambda(T)$, $T$ a triangulation of a surface, 
$M$ is Gorenstein projective 
if and only if $M = \Omega^1(N)$ for some $N$. 
Since finite dimensional Jacobian algebras of quivers with potential  
are 2-CY tilted algebras, this also follows from~\cite{GarciaESch2017}.

%
\subsection{Extensions between string modules}\label{sec:string-exts}

Let $\Lambda=KQ/I$ be gentle, let $v$ and $w$ be strings and $M(v)$ and $M(w)$ 
the corresponding string modules. Theorem A of \cite{CPS17}, see also \cite{CS17}, describes the extensions 
of $M(v)$ by $M(w)$. More precisely, it proves that the collection of arrow and overlap extensions 
of $M(v)$ by $M(w)$, as defined below, form a basis of $\Ext^1_{\Lambda}(M(v),M(w))$. 
If there is an arrow $\alpha$ 
such that 
$w\alpha^{-1}v$ is a string, then there is a non-split short exact sequences 
\[
0\to M(w)\to M(w\alpha^{-1}v)\to M(v)\to 0. 
\]
Such an extension is called an {\em arrow extension} of $M(v)$ by $M(w)$. 
See Figure~\ref{fig:arrow-ext} for a representation in terms of strings. 
The {\em overlap extensions} of $M(v)$ by $M(w)$ are described in terms of strings 
in Figure~\ref{fig:overlap-ext}, we refer to \cite{CS17, CPS17} for more details. 

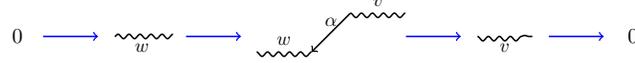
\begin{figure}[H]
\[ \resizebox{0.6\textwidth}{!}{
\begin{tikzpicture}[node distance=1cm and 1.5cm]
\coordinate[label=left:{}] (1);
\coordinate[right=1cm of 1] (2);

\coordinate[right=1.3cm of 1] (3);
\coordinate[right=1cm of 3] (4);

\coordinate[below right=.4cm of 4] (5);
\coordinate[right=1cm of 5] (6);

\coordinate[above right=1cm of 6] (7);
\coordinate[right=1cm of 7] (8);

\coordinate[right=3cm of 4] (9);
\coordinate[right=1cm of 9] (10);

\coordinate[right=.3cm of 10] (11);
\coordinate[right=1cm of 11] (12);

\coordinate[right=-.3cm of 1] (1');
\coordinate[right=-1cm of 1'] (2');
\coordinate[right=-.7cm of 2',label=right:{$0$}] (3');

\coordinate[right=.3cm of 12] (12');
\coordinate[right=1cm of 12'] (13');
\coordinate[right=.3cm of 13',label=right:{$0$}] (14');

\draw[thick,->,blue] (3)--(4);
\draw[thick,->,blue] (9)--(10);
\draw[thick,->,blue] (2')--(1');
\draw[thick,->,blue] (12')--(13');

\draw[thick,decorate,decoration={snake,amplitude=.4mm,segment length=2mm}] 
(1)-- node [anchor=north,scale=.9]{$w$} (2) 
(5)-- node [anchor=south,scale=.9]{$w$} (6) 
(7)-- node [anchor=south,scale=.9]{$v$} (8) 
(11)-- node [anchor=north,scale=.9]{$v$} (12);

\draw[thick,<-] 
(6)-- node [anchor=south,scale=.9]{$\alpha$} (7);
\end{tikzpicture} }
\]
\caption{An arrow extension of $M(v)$ by $M(w)$.} \label{fig:arrow-ext}
\end{figure}

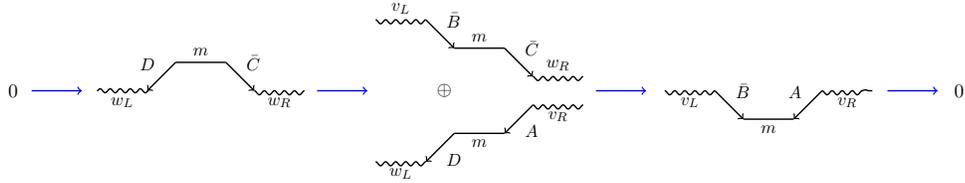
\begin{figure}[H]
\[ \resizebox{0.9\textwidth}{!}{
\begin{tikzpicture}[node distance=1cm and 1.5cm]
\coordinate[label=left:{}] (1);
\coordinate[right=1cm of 1] (2);
\coordinate[above right=.8cm of 2,label=right:{}] (3);
\coordinate[right=1cm of 3] (4);
\coordinate[below right=.8cm of 4] (5);
\coordinate[right=1cm of 5] (6);
\coordinate[above right=2cm of 6] (7);
\coordinate[right=1cm of 7] (8);
\coordinate[below right=.8cm of 8] (9);
\coordinate[right=1cm of 9] (10);
\coordinate[below right=.8cm of 10] (11);
\coordinate[right=1cm of 11] (12);
\coordinate[right=11.3cm of 1] (13);
\coordinate[right=1cm of 13] (14);
\coordinate[below right=.8cm of 14] (15);
\coordinate[right=1cm of 15] (16);
\coordinate[above right=.8cm of 16] (17);
\coordinate[right=1cm of 17] (18);

\coordinate[right=.25cm of 6] (6'');
\coordinate[right=1cm of 6''] (7'');

\coordinate[right=2.5cm of 6,label=right:{$\oplus$}] (19);

\coordinate[below right=2cm of 6] (7');
\coordinate[right=1cm of 7'] (8');
\coordinate[above right=.8cm of 8'] (9');
\coordinate[right=1cm of 9'] (10');
\coordinate[above right=.8cm of 10'] (11');
\coordinate[right=1cm of 11'] (12');
\coordinate[above right=.8cm of 12'] (13');

\coordinate[right=5.8cm of 6] (12'');
\coordinate[right=1cm of 12''] (13'');

\coordinate[right=-.3cm of 1] (1');
\coordinate[right=-1cm of 1'] (2');
\coordinate[right=-.6cm of 2',label=right:{$0$}] (3');

\coordinate[right=.3cm of 18] (18');
\coordinate[right=1cm of 18'] (19');
\coordinate[right=.2cm of 19',label=right:{$0$}] (19''');

\draw[thick,->,blue] (2')--(1');
\draw[thick,->,blue] (6'')--(7'');
\draw[thick,->,blue] (12'')--(13'');
\draw[thick,->,blue] (18')--(19');

\draw[thick,decorate,decoration={snake,amplitude=.4mm,segment length=2mm}] 
(1)-- node [anchor=north,scale=.9]{$w_L$} (2) 
(5)-- node [anchor=north,scale=.9]{$w_R$} (6) 
(7)-- node [anchor=south,scale=.9]{$v_L$} (8) 
(7')-- node [anchor=north,scale=.9]{$w_L$} (8') 
(11)-- node [anchor=south,scale=.9]{$w_R$} (12) 
(11')-- node [anchor=north,scale=.9]{$v_R$} (12') 
(13)-- node [anchor=north,scale=.9]{$v_L$} (14) 
(17)-- node [anchor=north,scale=.9]{$v_R$} (18);

\draw[thick] 
(3)-- node [anchor=south,scale=.9]{$m$} (4)
(9)-- node [anchor=south,scale=.9]{$m$} (10)
(9')-- node [anchor=north,scale=.9]{$m$} (10')
(15)-- node [anchor=north,scale=.9]{$m$} (16);

\draw[thick,->] (3)--node [anchor=south east,scale=.9]{$D$}(2); 
\draw[thick,->] (4)--node [anchor=south west,scale=.9]{$\bar{C}$}(5);
\draw[thick,->] (8)--node [anchor=south west,scale=.9]{$\bar{B}$}(9);
\draw[thick,<-] (8')--node [anchor=north west,scale=.9]{$D$}(9');
\draw[thick,->] (10)--node [anchor=south west,scale=.9]{$\bar{C}$}(11);
\draw[thick,<-] (10')--node [anchor=north west,scale=.9]{$A$}(11');
\draw[thick,->] (14)--node [anchor=south west,scale=.9]{$\bar{B}$}(15);
\draw[thick,->] (17)--node [anchor=south east,scale=.9]{$A$}(16);
\end{tikzpicture} }
\]
\caption{An overlap extension of $M(v)$ by $M(w)$ with overlap $m$.} \label{fig:overlap-ext}
\end{figure}

%
\section{Modules generated by an arrow}
%

Let $\Lambda=KQ/I$ be a gentle algebra. In this section we give a detailed study of modules generated by arrows by recalling and expanding
the results  from \cite{kalck} where it is shown that certain modules generated by an arrow are Gorenstein projective. We will see in the next section that in general the modules generated by an arrow constitute the non-projective summands of syzygies over gentle algebras. 

For $\alpha\in Q_1$,  set $R(\alpha)=\alpha\Lambda$. 
Furthermore, define $U(\alpha) = D(\Lambda \alpha)$.  
Since $\Lambda$ is gentle, $R(\alpha)$ and $U(\alpha)$ are  uniserial. 

\begin{defi}\label{def:cycles-n}
For $\Lambda=KQ/I$ gentle, denote by 
$\C(\Lambda)_n$ the set of arrows in cycles of length $n$ with full relations, i.e. for 
any two consecutive arrows $\alpha,\beta$ in such a cycle, $\alpha\beta\in I$. 
We define 
$\C(\Lambda)= \cup_{n>0} \C(\Lambda)_n$.  
\end{defi}

\begin{rem}\label{rem:Ralpha}
Let $\alpha$ be in  $ Q_1$. \\
(1) If there is no arrow $\beta$ such that $\alpha\beta\notin I$,  we have 
$R(\alpha)=M(e_{t(\alpha)}) = S(e_{t(\alpha)})$. 
Otherwise, $R(\alpha)=M(p)$ where $p$ is the maximal direct path 
starting with the unique arrow $\beta$ 
such that $\alpha\beta\notin I$. \\
Similarly,  if there is no arrow $\gamma$ such that $\gamma \alpha \notin I$, $U(\alpha) = M(e_{s(\alpha)}) = S(e_{s(\alpha)})$. 
Otherwise, $U(\alpha)=M(q)$ where $q$ is the maximal direct path 
ending with the unique arrow $\gamma$ 
such that $\gamma \alpha \notin I$. \\
(2) Consider $P_{t(\alpha)}$ and $I_{s(\alpha)}$ for $\alpha \in \C(\Lambda)$. 
Then there is a cycle $\alpha_0\alpha_1\cdots \alpha_n$ where $\alpha=\alpha_0$, 
such that $\alpha_i\alpha_{i+1}\in I$ for all $i$ (reducing modulo $n$). 
If there is no $\epsilon \in Q_1 $ such that $\alpha \epsilon \notin I$ then 
$P_{t(\alpha)} = R(\alpha)$. 
Otherwise $P_{t(\alpha)} $ is biserial, that is 
$P_{t(\alpha)} = M(q p)$ 
for $q$ an inverse string ending in $\alpha_1^{-1}$ and $p$ is a string as in (1). 
If there is no $\delta \in Q_1$ such that $\delta \alpha \notin I$ then 
$I_{s(\alpha)} = U(\alpha)$. Otherwise $I_{s(\alpha)}$ is biserial, that is  $I_{s(\alpha)} = M(r s)$ for $r$ a string as in (1) and $s$ an inverse string 
starting with $\alpha_n^{-1}$.  \\
(3) $R(\alpha)$ is a direct summand of the radical of the projective 
$P_{s(\alpha)}=e_{s(\alpha)}\Lambda$ 
(cf.~\cite{kalck}) and  $U(\alpha)$ is a direct summand of $I_{t(\alpha)} / \soc I_{t(\alpha)}$.
\end{rem}

\begin{lemma}\label{lm:R-alpha} Let $\Lambda=KQ/I$ be a gentle algebra 
and let $\alpha$ be in $Q_1$. \\
(1) $R(\alpha)$ is projective if and only if 
there is no $\beta \in Q_1$ such that $\alpha \beta \in I$;\\
(2) 
Suppose that there exist
$\alpha_1, \alpha_2, \ldots, \alpha_n$, all distinct, where $\alpha_1 = \alpha$ is such that 
$\alpha_i \alpha_{i+1} \in I$ for all $i<n$.
Then there are short exact sequences  
\[
\begin{array}{c}
0\to R(\alpha_2)\to P(R(\alpha_1)) \to R(\alpha_1)\to 0 \\
0\to R(\alpha_3)\to P(R(\alpha_2)) \to R(\alpha_2)\to 0 \\
\vdots \\
0\to R(\alpha_n)\to P(R(\alpha_{n-1})) \to R(\alpha_{n-1})\to 0
\end{array}
\]

Furthermore, $\pdim R(\alpha)=n$ if and only if 
there exists no arrow $\beta$ with 
$\alpha_n\beta\in I$ and $\pdim R(\alpha)=\infty$ if and only if 
$\alpha \in \C(\Lambda)$. 
\end{lemma}

\begin{proof}
(1) Suppose first that $R(\alpha)$ is projective, that is, $R(\alpha)=P_{t(\alpha)}$. 
Suppose further that there is $\beta \in Q_1$ such that $\alpha \beta \in I$. 
Then by Remark~\ref{rem:Ralpha} (2), $P_{t(\alpha)} = M(w)$ where $w$ 
is a string $ q \beta^{-1} p$ where $q$ is a possibly zero inverse string such 
that $q \beta^{-1} \notin I$ and where $p$ is a possibly zero direct string.  
Suppose first that there exists 
no $\gamma \in Q_1$ such that $\alpha \gamma \notin I$ then 
$R(\alpha) = S_{t(\alpha)}$, that is,
$R(\alpha)$ is the simple module at $t(\alpha)$. But $S_{t(\alpha)}$ is only 
projective if it is 
the projective at $t(\alpha)$, a contradiction. 
Suppose now that there exists $\gamma$ such that $\alpha \gamma \notin I$. 
Let 
$\gamma_1 \cdots \gamma_n$ be the maximal sequence of arrows 
where $\gamma=\gamma_1$ 
and $\gamma_i \gamma_{i+1} \notin I$. 
Then $R(\alpha)$ is the 
uniserial corresponding to this sequence and it is not projective since the string 
of $P_{t(\alpha))}$ has as substring $\beta^{-1} \gamma_1 \ldots \gamma_n$. 
Thus we have shown that if $R(\alpha)$ is projective then there is no 
$\beta \in  Q_1$ such that $\alpha \beta \in I$.
\\
The converse is immediate, since 
if $R(\alpha)$ is not projective, there exist $\beta$ such that $\alpha\beta\in I$. 
\\
(2) By definition we have that $\alpha_1 \alpha_2 \in I$. 
Suppose first that there exists no $\epsilon \in Q_1$ such that $\alpha_1 \epsilon \notin I$. Then by 
Remark~\ref{rem:Ralpha} (1) and (2), 
$R(\alpha_1) = M(e_{t(\alpha_1)})$ and $P_{t(\alpha_1)} = M(\alpha_2 q)$ where $q$ is the maximal direct string such that $\alpha_2 q\notin I$. 
Hence $\Omega^1(R(\alpha_1)) = M(q) = R(\alpha_2)$. Suppose now that there exists 
$\epsilon \in Q_1$ such that $\alpha_1 \epsilon \notin I$. Then by Remark~\ref{rem:Ralpha} (1), $R(\alpha_1) = M(p)$ where 
$p$ is the unique maximal direct string such that $\alpha_1 p \notin I$ and by 
Remark~\ref{rem:Ralpha} (2), 
$P_{t(\alpha_1)} = M(q^{-1} \alpha_2^{-1} p)$. 
Hence we have again that 
$\Omega^1(R(\alpha_1)) = M(q) = R(\alpha_2)$ and if there exists no 
$\beta$ such that $\alpha_n\beta \in I$, it follows that $\pdim R(\alpha)=n$. 

If $\pdim R(\alpha)=n$, then $R(\alpha_n)$ is projective and it follows from (1) that there exists no 
arrow $\beta$ with $\alpha_n\beta\in I$. 

If any of the $\alpha_i$ are in $\C(\Lambda)$ then since $\Lambda $ is gentle, $\alpha \in \C(\Lambda)$ 
and $R(\alpha_m) = R(\alpha)$ for some $m$. On the other hand if 
$\pdim R(\alpha)=\infty$ then since $\Lambda$ is finite dimensional, there exits some $m> n$ such that 
$R(\alpha_m)=\Omega^{m-1} R(\alpha_1)$ and $\alpha_m=\alpha_1$. Hence $\alpha = \alpha_1 \in C(\Lambda)$. 
\end{proof}

\begin{lemma}\label{lm:U-version}
Let $\Lambda=KQ/I$ be a gentle algebra 
and let $\alpha$ be in $Q_1$. \\
(1) $U(\alpha)$ is injective if and only if 
there is no $\gamma \in Q_1$ such that $\gamma \alpha  \in I$;\\
(2) Suppose that there exist 
$\alpha_1,\dots, \alpha_n$, all distinct, where $\alpha_1=\alpha$ and 
where $\alpha_i\alpha_{i+1}\in I$ for all $i<n$ 
then there are short exact sequences
\[
\begin{array}{c}
0\to U(\alpha_n)\to I(U(\alpha_{n})) \to U(\alpha_{n-1})\to 0 \\
0\to U(\alpha_{n-1})\to I(U(\alpha_{n-1})) \to U(\alpha_{n-2})\to 0 \\
\vdots \\
0\to U(\alpha_2)\to I(U(\alpha_{2})) \to U(\alpha_1)\to 0. 
\end{array}
\]
Furthermore, $\idim U(\alpha)=n$ if and only if there exists no arrow $\gamma$ with 
$\alpha_n\gamma\in I$ and $\idim U(\alpha)=\infty$ if and only if 
$\alpha_i\in \C(\Lambda)$ for some $i$. 
\end{lemma}

The proof of Lemma~\ref{lm:U-version} is analogous to the proof of Lemma~\ref{lm:R-alpha}.

\begin{rem}
Note that if $M$ is an injective module over a gentle algebra $KQ/I$, 
i.e. $M=M(pq^{-1})$ for $p,q$ direct paths in $Q$, the syzygies 
$\Omega^i(M)$ have been determined for all $i$ in~\cite{gr2005}.
\end{rem}

The following is an immediate consequence of Lemma~\ref{lm:R-alpha} and~\ref{lm:U-version} and their proofs. 

\begin{cor}\label{cor:Ralpha-cyclic}
If $\alpha$ is in an $n$-cycle with full relations, then 
$\Omega^n(R(\alpha))=R(\alpha)$. 
\end{cor}

\begin{defi}
If $w$ is a direct string such that there exists no $\alpha\in Q_1$ such that $w\alpha\notin I$ (respectively, 
$\alpha w \notin I$), we 
say that $w$ is a {\em right (respectively, left) maximal string}. Similarly, we define left and right maximal 
for inverse strings. 
For $R(\alpha)=M(p)$ with $p$ right maximal, we call $p$ the {\em string of $R(\alpha)$}. 
\end{defi}

\begin{prop}\label{prop:arrow-ext}
Let $\Lambda = KQ/I$ be gentle and let $M = M(w)$ be an indecomposable string module over $\Lambda$, 
$\alpha\in Q_1$ with $p$ the string of $R(\alpha)$. Then 
$\Ext_\Lambda^1(R(\alpha),M)\ne 0$ if and only if there exists an arrow $\beta$ 
such that $\alpha\beta\in I$ and $w\beta^{-1}p$ is a non-zero string. \\
In that case, $\Ext_\Lambda^1(R(\alpha),M)$ is one-dimensional and 
is the arrow extension given by $\beta$. 
\end{prop} 

\begin{rem}
Note that if $\Lambda=\LT$ 
in Proposition~\ref{prop:arrow-ext}, then 
$\Ext_\Lambda^1(R(\alpha),M) \neq 0$ implies that 
$\alpha$ is in $\C(\Lambda)$. 
\end{rem}

\begin{proof}[Proof of Proposition~\ref{prop:arrow-ext}]
By \cite[Theorem A]{CPS17}, if $\Ext_\Lambda^1(R(\alpha),M)\ne 0$, these 
extensions are arrow or overlap extensions. 
We show first that there can be no overlap extensions. So suppose 
there exists an overlap extension with overlap $m$, see Figure~\ref{fig:overlap-ext}. 
Then $p=v_Lm$ is a direct string and $v_R$ is trivial. 
Since $p$ is right maximal, $m$ is right maximal and therefore, 
$w_R$ is trivial. But this contradicts the fact there is an overlap extension as 
both $w$ and $p$ end in the overlap in this case. 

Therefore, any extension must be an arrow extension with an arrow $\gamma$ such that 
$w\gamma^{-1}p$ is a non-zero string, \cite[Theorem A]{CPS17}.

If there exists no $\beta$ such that $\alpha\beta\in I$ then $R(\alpha)$ is projective 
(Lemma~\ref{lm:R-alpha} (1)) 
and $\Ext_\Lambda^1(R(\alpha),M)=0$. So assume there exists $\beta$ 
with $\alpha\beta\in I$. 
Then 
$\Ext_\Lambda^1(R(\alpha),M)$ is an arrow extension 
if and only if 
$w\beta^{-1}$ is a non-zero string, \cite[Theorem A]{CPS17}. 
This completes the first part. 

To see that $\dim\Ext_\Lambda^1(R(\alpha),M)=1$, note that 
there must be at least one arrow extension with some arrow $\gamma$. 
In particular, $s(p)=s(\gamma)$. 
Since $\alpha\beta\in I$, $\beta$ is not an arrow in $p$ and as $\Lambda$ is gentle, 
$\gamma=\beta$. 
\end{proof}

\begin{cor}\label{cor:self-ext} 
Let $\Lambda = KQ/I$ be a gentle algebra and let $\alpha\in Q_1$. Then \sloppy
$\dim \Ext_\Lambda^1(R(\alpha),R(\alpha))\le 1$. 
In particular, if  $\pdim R(\alpha) \leq 2$ or 
if $\Lambda=\LT$ for $T$ a triangulation of a surface then
$R(\alpha)$ is rigid. 
\end{cor}

\begin{proof} 
The first part is Proposition~\ref{prop:arrow-ext}.  Let $R(\alpha)=M(p)$. 
If $R(\alpha)$ has an arrow extensions with arrow $\beta$, then 
$p\beta^{-1}p$ is a string. This implies that $\alpha \beta \in I$. Now suppose that $\Lambda = \LT$. 
Then if $\alpha \notin \C(\Lambda)$, $R(\alpha)$ is projective. So suppose
that $\alpha \in \C(\Lambda)$.  Since $\alpha$ is in a full cycle of relations, there exists an arrow 
$\gamma$ such that $\beta \gamma \in I$. Since $\Lambda$ is gentle, we have $p \gamma \notin I$, 
contradicting the right maximality of $p$.  Similarly if $\Lambda$ is general and $\pdim R(\alpha) \geq 2$, 
there exists $\gamma$ such  $\beta \gamma \in I$ and $p \gamma \notin I$, contradicting the right maximality of $p$.
\end{proof}

%
\section{Higher syzygies}\label{sec:syzygies}

In this section, we determine syzygies of indecomposable modules over gentle algebras and 
deduce some consequences for higher extensions. 
Since gentle algebras are monomial and hence string algebras, it already follows from~\cite{HZ91, HZS05} that in a projective resolution, 
from the second syzygy onwards, all indecomposable summands are uniserial. 
However, we need a more precise description of the  decomposition of the syzygies into direct sums of indecomposable modules.

\begin{prop}\label{prop:syzygy}
Let $M$ be an indecomposable string module over 
a gentle algebra $\Lambda = KQ/I$. Then we have: \\
(1) 
$\Omega^1(M)=L_0\oplus \cdots \oplus L_k$ for some $k$ and where for $i=1,\dots, k-1$, 
$L_i$ is projective 
and $L_0$ and $L_k$ are uniserial or zero. 

\noindent
Furthermore, either $\Omega^i(M)=0$ for 
some $i>0$ or there exists some $n>0$ such that 
$\Omega^{ln+i}(M)=\Omega^i(M)$ for $i\ge m$ for some positive integers $m$ and for $l \geq 0$. 
 \\
(2)
$\Omega^{-1}(M)=N_0\oplus \cdots \oplus N_l$ for some $l$ and where for $i=1,\dots, l-1$, 
$N_i$ is injective
and $N_0$ and $N_l$ are uniserial or zero. 

\noindent
Furthermore, either $\Omega^{-i}(M)=0$ for 
some $i>0$ or there exists some $n>0$ such that 
$\Omega^{-ln-i}(M)=\Omega^{-i}(M)$ for $i\ge m$ for some positive integer $m$ and for $l \geq 0$. 
\end{prop}

\begin{proof} 
We will prove part (1) of the proposition.  Part (2) then follows by a dual argument. 
Let $M=M(w)$ where $w=w_1 w_2 \cdots w_n$ where each $w_i$ is a direct 
or inverse string, alternatingly. 
The string $w$ of $M$ can be presented in the following way
\[
\includegraphics[width=10cm]{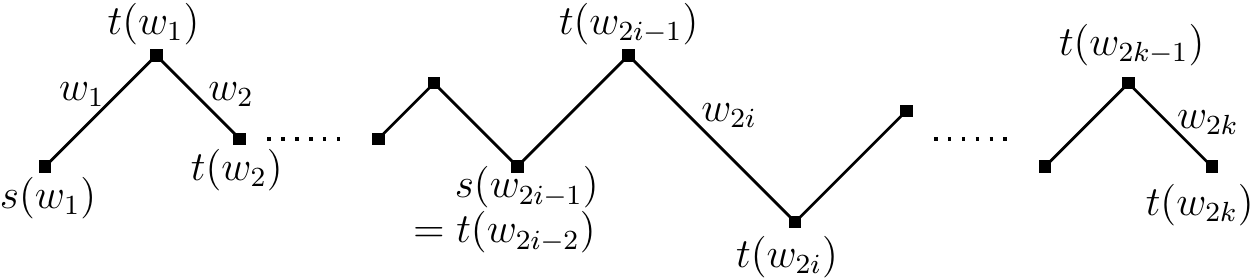}
\]
where if $w$ starts (resp. ends) with a direct (resp. inverse) string, 
$w_1$ (resp. $w_{2k}$) is trivial. 

From now on we will always allow that $w_1$ or $w_{2k}$ or both are trivial and we will only explicitly treat these 
cases separately when necessary. Note that if $w_1$ (resp.  $w_{2k}$) is trivial then 
$s(w_1) = t(w_1) = s(w_2)$ (resp. $s(w_{2k})=t(w_{2k}) = t(w_{2k-1})$). 
The top of $M$ is 
\[
\tp(M)=S(t(w_1)) \oplus S(t(w_3)) \oplus\cdots \oplus S(t(w_{2k-1}))
\]
and the socle of $M$ is
\[
\soc(M)=S(s(w_1))\oplus S( t(w_2)) \oplus S(t(w_4)) \oplus\cdots \oplus S(t(w_{2k}))
\] 
where we adopt the convention that $S(s(w_1)) =  0$ (resp. $S(t(w_{2k})) = 0$)  if $w_1$ (resp. $w_{2k}$) is trivial. 
It follows from the above that we have 
\[
P(M)=P(t(w_1))\oplus P(t(w_3))\oplus\cdots 
\oplus P(t(w_{2k-1})).
\] 
Pictorially, $P(M)$ looks as follows. 
The red strings indicate summands of the first syzygy, the white 
square boxes indicate two-dimensional vector spaces: for $i=1,\dots, k$, the simple 
$t(w_{2i})$ appears in the two projectives $P(t(w_{2i-1}))$ and $P(t(w_{2i+1}))$.  
\[
\includegraphics[width=12cm]{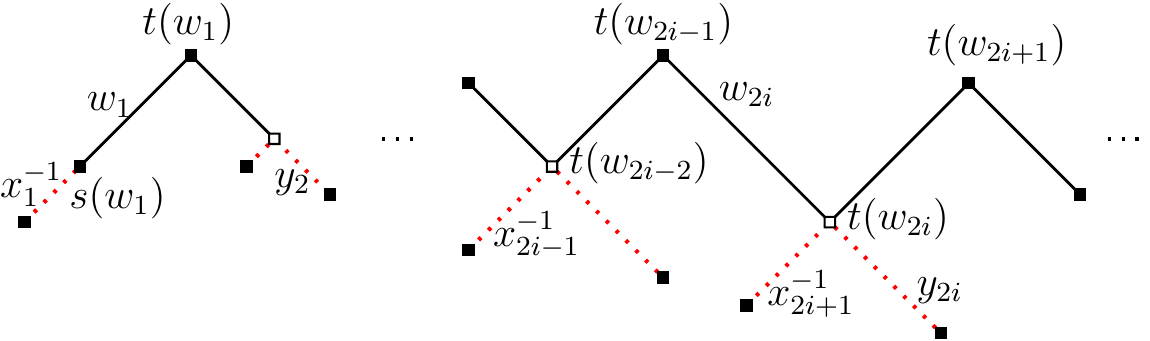}
\]
Note that $P(t(w_{2i-1}))= x^{-1}_{2i-1}w_{2i-1}w_{2i}y_{2i}$ where 
$x^{-1}_{2i+1}$ and $y_{2i}$ might be trivial strings, for $1 \le i \le (n+1)/2$ and where $x^{-1}_{2i-1}$  
(resp. $y_{2i}$) is the unique string such that $x^{-1}_{2i-1}w_{2i-1}$ (resp. $w_{2i}y_{2i}$) is a left 
maximal inverse (resp. right maximal direct) string. If  $w_1$ is trivial then 
$P(t(w_{1}))= x^{-1}_{1}w_{2}y_{2}$ and if  $w_{2k}$ is trivial $P(t(w_{2k-1}))= x^{-1}_{2k-1}w_{2k-1}y_{2k}$.

We will now determine the summands of 
$\Omega^j(M)$, for $j \ge 1$ beginning with $\Omega^1(M)$. 
The summands of $\Omega^1(M)$ arise from the socle of $M$. In (i) 
we consider the summands corresponding to the ``interior'' of the string $w$, in (ii) and (iii) the summands 
arising from the beginning and the end of $w$.

(i) We start by determining $L_i$ for  $1\le i< k$. 
Suppose first that both $x_{2i+1}^{-1}$ and $y_{2i}$ are non-trivial. 
Note that in this case, 
$x_{2i+1}^{-1}y_{2i}$ is a string and the string module $M(x_{2i+1}^{-1}y_{2i})$ is a direct summand 
of $\Omega^1(M)$ and we have $L_i = M(x_{2i+1}^{-1}y_{2i})$. 
Since $\Lambda$ is gentle, by maximality of $x_{2i+1}^{-1} w_{2i+1}$ and of $w_{2i} y_{2i}$, 
$L_i$ is the projective module 
$P(t(w_{2i}))$. \\
Suppose now that only one of $x_{2i+1}^{-1}$ and $y_{2i}$ is non-trivial. 
Suppose first that $x_{2i+1}^{-1}$ is trivial  
and  that $y_{2i} = \beta_1 \ldots \beta_m$. Then we must have 
$L_i = P(t(w_{2i})) = M(y_{2i})$. 
Similarly, if only $x_{2i+1}^{-1}$ is non-trivial then $L_i = P(t(w_{2i})) = M(x_{2i+1}^{-1})$.
If both  $x_{2i-1}$ and $y_{2i}$ are trivial then 
$L_i = P(t(w_{2i}))$ is the simple module at $t(w_{2i})$ and it is projective. 
Thus in all cases $L_i =P(t(w_{2i}))$.

(ii) We now determine $L_0$.  
If $x_1^{-1}$ is trivial 
then $L_0=0$. 
If $x_1^{-1}$ is not trivial then 
$ x_1^{-1} = (\alpha_m \ldots \alpha_1)^{-1}$ for $\alpha_1, \ldots, \alpha_m \in Q_1$. 
Then clearly $L_0=R(\alpha_1)$ which is uniserial.
Following Lemma~\ref{lm:R-alpha} (1)  
 there are two possible cases depending on whether 
$\alpha_1$ is in $\C(\Lambda)$ or not: Suppose that $\alpha_1\notin \C(\Lambda)$. Let 
$\gamma_1 \ldots \gamma_m$ be the maximal path in $KQ$ such that $\alpha_1 = \gamma_1$ 
and $\gamma_i \gamma_{i+1} \in I$. Then $L_0$ has projective dimension $m$ and 
$\Omega^i (L_0)$ is a direct summand of $\Omega^{i+1}(M)$  
for all $ 1 \leq i \leq m-1$. In particular, we have $\Omega^i(M) = 0$ for $i \geq m$.
Suppose now that  $\alpha_1\in \C(\Lambda)_n$ for some $n \geq 1$, that is, there are arrows 
$\gamma_1, \ldots, \gamma_n$ in $Q_1$ such that $\gamma_1=\alpha_1$, 
$\gamma_i \gamma_{i+1} \in I$ and $\gamma_n \gamma_1 \in I$.    Then by Lemma~\ref{lm:R-alpha}, 
$\Omega^{ln+i} R(\alpha_1) = \Omega^i R(\alpha_1)$ for $ 0 \leq i \leq n-1$ and for all $l \geq 0$ 
and $R(\alpha_1)$ is a direct summand of $\Omega^{ln+1}(M)$. More generally in this case, 
$\Omega^{nl+i} (R(\alpha_1)) = R(\alpha_i)$ is a direct summand of $\Omega^{nl+i+1} (M)$ 
for all $i \geq 1$. 

(iii) 
This is analogous to case (ii). 
\end{proof}

\begin{rem}\label{rem:L-Ralpha}
In the proof of Proposition~\ref{prop:syzygy} we give an explicit 
description of the strings defining the modules $L_i$ and $N_i$. 
In particular, keeping the notation of Proposition~\ref{prop:syzygy}, we see 
that if $L_0$ and $L_k$ are non-zero then they are of the form $R(\alpha)$ for some arrow $\alpha\in Q_1$. 
Similarly, if $N_0$ and $N_l$ are non-zero then they are of the form $U(\beta)$ for some $\beta\in Q_1$. 
\end{rem} 

Proposition~\ref{prop:syzygy} combined with Lemma~\ref{lm:R-alpha}  yields the following. 

\begin{cor}\label{cor:ext2}
Let $M,N$ be indecomposable string modules over a gentle algebra $\Lambda$. Let 
$\Omega^1(M)=L_0\oplus\cdots \oplus L_k$ be as in Proposition~\ref{prop:syzygy}. Then 
\[
\Ext_\Lambda^2(M,N)=\Ext_\Lambda^1(L_0,N)\oplus \Ext_\Lambda^1(L_k,N).
\]
More generally, for $n\ge 2$ we have 
\[
\Ext_\Lambda^n(M,N)=\Ext_\Lambda^1(\Omega^{n-2}(L_0),N)\oplus 
\Ext_\Lambda^1(\Omega^{n-2}(L_k),N) 
\]
where for $i=0,k$, the module 
$\Omega^{n-2}(L_i)$ is either 0 or of the form $R(\alpha)$ for some arrow $\alpha$. \\
In particular, if $\Lambda=\LT$, for $n\ge 2$,  
\[
\Ext_\Lambda^{n}(M,N)=\Ext_\Lambda^{n+3}(M,N). 
\]
\end{cor}

By Corollary~\ref{cor:ext2} to compute higher extensions between $M$ and $N$ 
indecomposable string modules, 
it is enough to determine $\Ext_\Lambda^1(R(\alpha),N)$, for certain arrows $\alpha$. 

\medskip

We can use Proposition~\ref{prop:arrow-ext} and Proposition~\ref{prop:syzygy} to analyse the pattern of the higher 
syzygies between indecomposable string modules. For this, we need the following notion: Let $M(w)$ 
be a string module, $w=w_1 \cdots w_n$. We say that $M(w)$ or the string $w$ 
{\em minimally ends in a cycle} if at least one of the following is the case: 

(i) $w_1$ is inverse, there exists $\alpha\in\C(\Lambda)$ with $s(w_1)=t(\alpha)$ and $\alpha^{-1}w$ 
is a string. 

(ii) $w_n$ is direct, there exists $\beta\in \C(\Lambda)$ with $t(w_n)=s(\beta)$ and $w\beta$ is a string.

As a summary, using Proposition~\ref{prop:arrow-ext}, we now have the following. 

\begin{theorem}\label{thm:characterisation}
Let $\Lambda = KQ/I$ be a gentle algebra and $M, N$ two indecomposable $\Lambda$-modules. Then 
$\Ext^i_\Lambda(M,N)$ is either eventually zero or becomes periodic. 

Furthermore, $\Ext^i(M,N)$ is eventually periodic of period $k$ if and only if 
$M$ and $N$ minimally end in a common cycle, 
where $k$ is the size of the corresponding cycle of full relations 
or the least common multiple of the lengths of the cycles of full of relations at both ends of the strings. 

More precisely, there exists $n_0>0$ such that for all $i\ge n_0$ and for all $l\ge 0$, 
$\Ext^{i}_\Lambda(M,N) = \Ext^{i +kl}_\Lambda(M,N)$. 
\end{theorem}

The above implies that most of the higher extensions between indecomposable string modules 
are 0 and we have: 

\begin{cor}\label{cor:dimensions}
Let $M$ and $N$ be indecomposable string modules for a gentle algebra $\Lambda$. 
Set $m_j:=\dim \Ext_{\Lambda}^j(M,N)$. Then there exists $n_0>0$ such that for all 
$j\ge n_0$, $m_j+m_{j+1}+\ldots + m_{j+k-1}\le 2$, for $k$ as in Theorem~\ref{thm:characterisation}, 
and this number is constant. 

If $\Lambda=\Lambda(T)$, then $n_0=2$ and $k=3$. 
\end{cor}

The arguments in the proof of Proposition~\ref{prop:syzygy} applied to band modules give the following 
well-known result: 

\begin{cor}\label{cor:bands}
Let $\Lambda$ be a gentle algebra. 
Let $w$ be a band and $B(w)$ the associated quasi-simple band module over 
$\Lambda$. 
Then $pd(B(w)) = id(B(w))= 1$. 
\end{cor}

\begin{proof}
This follows directly from the analysis of $\Omega^1$ and $\Omega^{-1}$ in 
the proof of Proposition~\ref{prop:syzygy} together with the fact that in a band $w$ 
there is no beginning or end. 
\end{proof}

%
\section{Geometric interpretations of higher extensions} \label{sec:geometric}


In this section, we give geometric interpretations of the vanishing or periodicity of higher extensions, using surface models 
for gentle algebras. 
For gentle algebras arising from triangulations of surfaces using the geometric model in \cite{ABCP}, the dimension of $\Ext^1$ between two indecomposable 
string modules is bounded from above by the number of intersections between the two associated curves  \cite{CS17}. 
For the general case, using the explicit description of extensions in~\cite{CPS17} (see also~\cite{BDMTY17}) 
and the geometric model in~\cite{BCS2018}, the analogous result holds. 

For the higher extension spaces, we can use the results from Section~\ref{sec:syzygies} to give a geometric 
characterisation of higher extension spaces. Since it is possible to give a concise description of  all higher extensions in the 
case of gentle algebras arising from triangulations of surfaces we start with this case.

%
\subsection{Triangulations}\label{sec:triangulation} 
%

We begin by giving an explicit interpretation of bases of the higher extensions in the case of triangulations. 
For the remainder of Section~\ref{sec:triangulation}, 
$\Lambda=\LT$ is a gentle algebra arising from a triangulation $T$ of a surface. 
Higher non-trivial syzygies only arise when the two indecomposable string 
modules minimally end in at least one common 
3-cycle, see Theorem~\ref{thm:characterisation}. 
In terms of triangulations of surfaces, this means that the 
arcs corresponding to indecomposable modules end in 
at least one common internal triangle, see also Remark~\ref{rem:geometric-way} for the general case. So 
in order to 
find bases for higher extensions between indecomposable string $\Lambda$-modules $M=M(v)$ and $N=M(w)$, 
we have to study the relative 
positions of the two ends of the arcs $\gamma_v$ and $\gamma_w$ in a common internal triangle. 
This will provide us with descriptions of 
$\Ext^i(M,N)$ for $i\ge 2$. 
By Proposition~\ref{prop:syzygy}, if 
$M(v)$ is an indecomposable string module we have 
$\Omega^1(M(v))=L_0\oplus\dots\oplus L_k$. Recall that if $L_0$ (or $L_k$ respectively) 
is non-zero, then there exists an arrow $\alpha$ 
such that $L_0=R(\alpha)$ and similarly for $L_k$. 
If $R(\alpha)$ is non-projective, then $\alpha\in \C(\Lambda)$ and therefore, we have 
a set-up as in Figure~\ref{fig:configuration} where $p$ is the string of $R(\alpha)$, that is $R(\alpha) = M(p)$. 
We will use the notation as in Figure~\ref{fig:configuration} for the rest of this  subsection. 

\begin{figure}[H]
\[
\includegraphics[width=7cm]{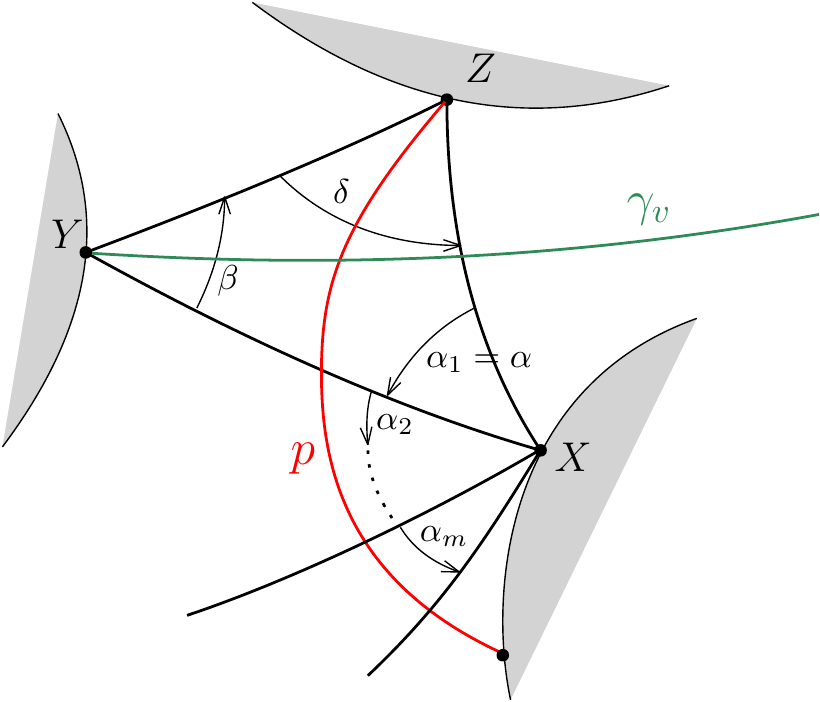}
\]
\caption{Arc $\gamma_v$ with $R(\alpha) = M(p)$ non-projective.}\label{fig:configuration}
\end{figure}

\begin{theorem}\label{thm:dim-ext2}
Let $\Lambda=\LT = KQ/I$ and let $M(v)$ and $M(w)$ be indecomposable string modules. 
Denote the corresponding arcs by $\gamma_v$ and $\gamma_w$. Then 
\[
\dim\Ext_\Lambda^2(M(v),M(w))=\left\{
\begin{array}{cl}
2 & \mbox{if $\gamma_v,\gamma_w$ are as in Case 1}\\
1 & \mbox{if $\gamma_v,\gamma_w$ are as in Case 2}\\ 
0 & \mbox{otherwise} 
\end{array}
\right.
\]
Case 1: 
At both ends   $\gamma_v$ follows $\gamma_w$  in the orientation of the surface in 
an internal triangle, see Figure~\ref{fig:v-follows-w}. \\
Case 2:  $\gamma_v$ follows $\gamma_w$ in the orientation of the surface only at one of the two ends of the arcs. \\
In case 1, the two arrow extensions given by $\beta$ and $\beta'$ as in 
Figure~\ref{fig:v-follows-w} induce a basis 
for $\Ext_\Lambda^2(M(v),M(w))$. In case 2, the basis is induced by only one of the arrow 
extensions. 

Furthermore, $\Ext^{2+3k}(M(v),M(w))\cong \Ext^2(M(v), M(w))$ for all $k\ge 0$.
\end{theorem}

\begin{proof}
By Corollary~\ref{cor:ext2}, 
$\Ext_\Lambda^2(M(v),M(w))=$ 
$\Ext_\Lambda^1(L_0,M(w)) \oplus \Ext_\Lambda^1(L_k,M(w))$. 
If $L_0$ respectively $L_k$ are non-zero, 
then $L_0=R(\alpha)$ and 
$L_k=R(\alpha')$ respectively, for some arrows $\alpha,\alpha'$ in $Q_1$.  

Suppose $\dim\Ext_\Lambda^2(M(v),M(w))\ne 0$. 
Then at least one of $R(\alpha)$ and $R(\alpha')$ is non-zero and 
non-projective. 
Suppose w.l.o.g. that $R(\alpha)$ is non-zero non-projective and furthermore, let $p$ 
be the right maximal string such that $R(\alpha) = M(p)$. 

Thus we are in the set-up of Figure~\ref{fig:configuration}. 
If $\gamma_w$ locally is as in Figure~\ref{fig:v-follows-w}, 
$w\beta^{-1}p$ is a string and therefore, 
by Proposition~\ref{prop:arrow-ext}, there is 
an arrow extension with arrow $\beta$. 
Thus $\dim\Ext_\Lambda^1(R(\alpha),M(w))=1$ and the result follows. 
\end{proof}

\begin{figure}[H]
\[
\includegraphics[scale=.8]{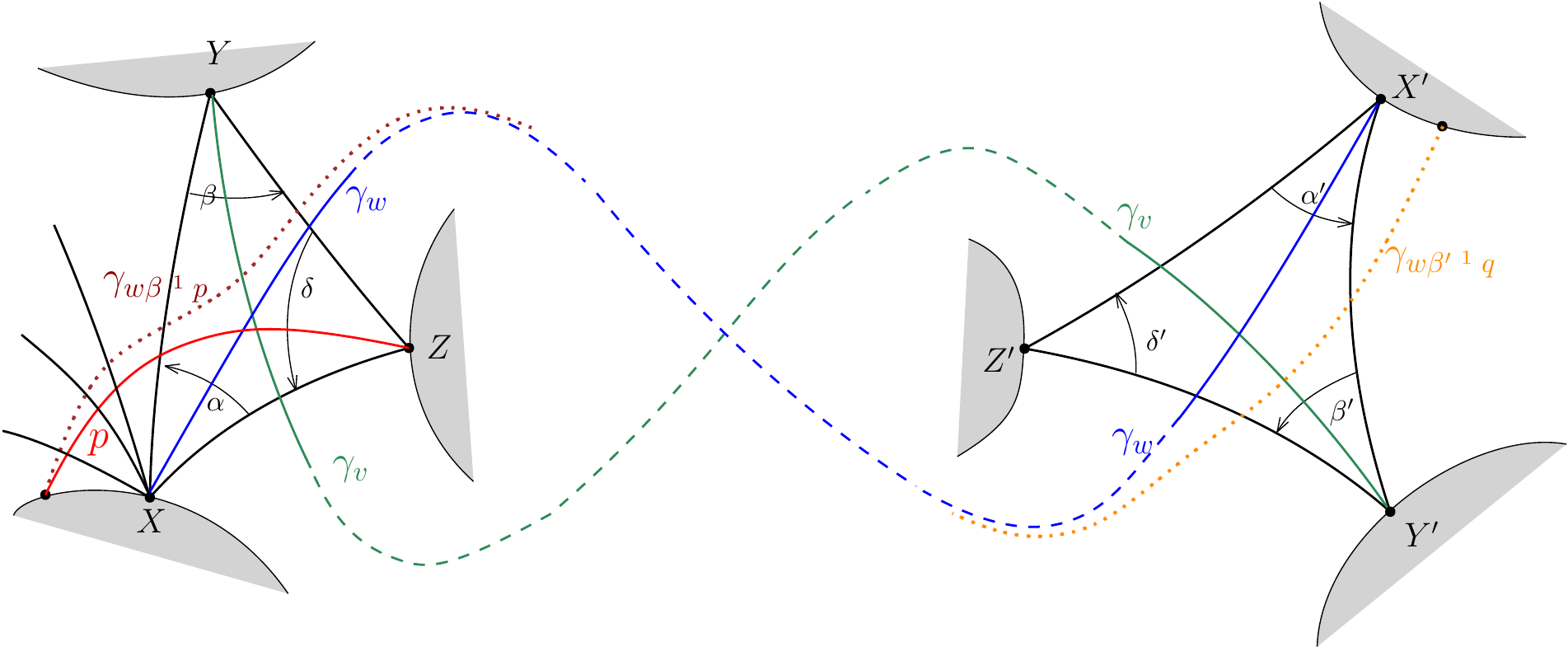}
\]
\caption{Local endpoint configurations of $\gamma_v$ and $\gamma_w$ 
contributing to $\Ext_\Lambda^2(M(v),M(w))$. Note that the 
vertices $\{X,Y,Z,X',Y',Z'\}$ are not necessarily distinct.} 
\label{fig:v-follows-w}
\end{figure}

\begin{rem}
We can describe a basis of $\Ext_\Lambda^2(M(v),M(w)) = 
\Ext_\Lambda^1(L_0,M(w)) \oplus \Ext_\Lambda^1(L_k,M(w))$ 
in terms of arcs.\\
The basis element corresponding 
to $\Ext_\Lambda^1(L_0,M(w))$ is given by the string $w\beta^{-1}p$ 
where $p$ is the string of $R(\alpha)$. 
Its arc $\gamma_{w\beta^{-1}p}$ starts at the anti clockwise neighbour of 
$X$ and follows $\gamma_w$ to end at $X'$. See left side of Figure~\ref{fig:v-follows-w}.  
Note that if $X$ is the only marked point on the boundary then $\gamma_{w\beta^{-1}p}$ is 
corresponds to the image of $\gamma_w$ under a Dehn twist. \\
The basis element corresponding to $\Ext_\Lambda^1(L_k,M(w))$ is given by the 
string $w\beta'^{-1} q$ where 
$q$ is the string of $R(\alpha')$. 
Its arc $\gamma_{w\beta'^{-1} q}$ starts at the anti clockwise neighbour of 
$X'$ (as for $X$ above this  might  be $X'$ itself) 
and follows $\gamma_w$ to end at $X$. See right side of Figure~\ref{fig:v-follows-w}.  
\end{rem}

\begin{theorem}\label{thm:ext3}
Let $\Lambda=\LT = KQ/I$ and let $M(v)$ and $M(w)$ be indecomposable string modules. 
Denote the corresponding arcs by $\gamma_v$ and $\gamma_w$. Then 
\[
\dim\Ext_\Lambda^3(M(v),M(w))=\left\{
\begin{array}{cl}
2 & \mbox{if $\gamma_v,\gamma_w$ are in Case 1}\\
1 & \mbox{if $\gamma_v,\gamma_w$ are in Case 2}\\ 
0 & \mbox{otherwise} 
\end{array}
\right.
\]
Case 1: both ends of $\gamma_v$ and $\gamma_w$ start at 
the same vertex in an internal triangle, cf.~Figure~\ref{fig:ext3}. \\
Case 2: $\gamma_v$ and $\gamma_w$ start at the same endpoint only at one side. \\
In case 1, the two arrow extensions given by $\delta$ and $\delta'$, as in 
Figure~\ref{fig:ext3},
induce a basis 
for $\Ext_\Lambda^3(M(v),M(w))$. In case 2, the basis is induced by only one of them. 

Furthermore, $\Ext^{3+3k}(M(v),M(w))\cong \Ext^3(M(v), M(w))$ for all $k\ge 0$.
\end{theorem}

\begin{figure}[H]
\[
\includegraphics[scale=.8]{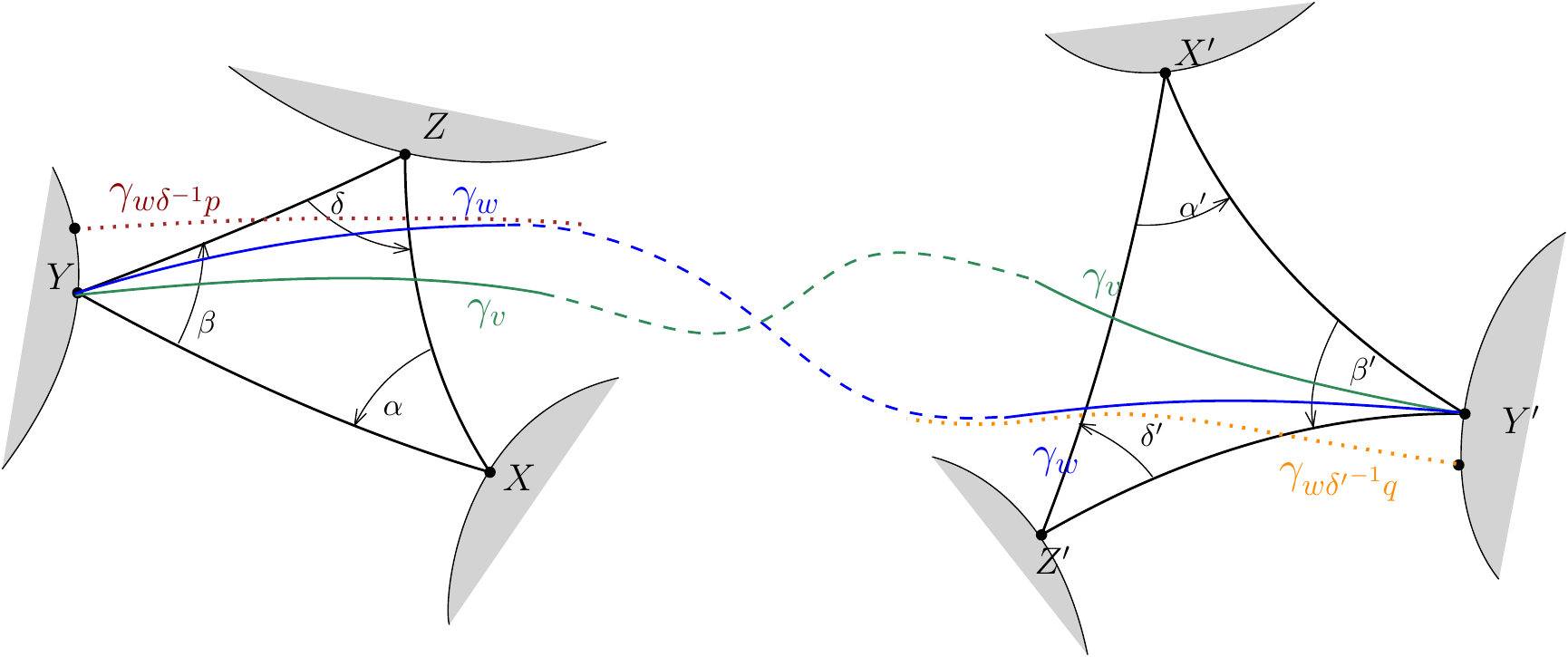}
\]
\caption{Local endpoint configurations of $\gamma_v$ and $\gamma_w$ 
contributing to $\Ext_\Lambda^3(M(v),M(w))$.} 
\label{fig:ext3}
\end{figure}

\begin{proof}
By Corollary~\ref{cor:ext2}, 
$\Ext_\Lambda^3(M(v),M(w))$
$=\Ext_\Lambda^1(\Omega^1(L_0),M(w))\oplus \Ext_\Lambda^1(\Omega^1(L_k),M(w))$. 
Recall that if $L_0$ and $L_k$ respectively are non-zero, then there exist arrows 
$\alpha$, $\alpha'$ such that $L_0=R(\alpha)$, $L_k=R(\alpha')$.

Suppose that $\Ext_\Lambda^3(M(v),M(w))\ne 0$, then at least one of $R(\alpha)$, 
$R(\alpha')$ is not projective. 
W.l.o.g.  let $R(\alpha)$ be non-projective. Then there exist $\beta$, $\delta$ such that 
$\alpha,\beta,\delta$ forms a triangle full of relations. 
By Lemma~\ref{lm:R-alpha} (3), $\Omega^1(R(\alpha))=R(\beta)$. 
Then by Proposition~\ref{prop:arrow-ext} $\Ext_\Lambda^1(\Omega^1(R(\alpha)),M(w))\ne 0$ 
if and only if $\gamma_v$ and 
$\gamma_w$ start at $X$, see Figure~\ref{fig:ext3} 
\end{proof}

\begin{rem}
We can describe a basis of $\Ext_\Lambda^3(M(v),M(w))$ in terms of arcs. \\
The basis element corresponding 
to $\Ext_\Lambda^1(\Omega^1(L_0,M(w)))$ is given by the string $w\delta^{-1}p$ 
where $p$ is the string of $R(\beta)$. 
Its arc $\gamma_{w\delta^{-1}p}$ starts at the anti clockwise neighbour of 
$Y$ and follows $\gamma_w$ to end at $Y'$. See left side of Figure~\ref{fig:ext3}. \\
The basis element corresponding to $\Ext_\Lambda^1(\Omega^1(L_k),M(w))$ 
is given by the 
string $w\delta'^{-1} q$ where 
$q$ is the string of $R(\beta')$. 
Its arc $\gamma_{w\delta'^{-1} q}$ starts at the anti clockwise neighbour of 
$Y'$ and follows $\gamma_w$ to end at $Y$. See right side of Figure~\ref{fig:ext3}. 
\end{rem}

\begin{theorem}\label{thm:ext4}
Let $\Lambda=\LT = KQ/I$ and let $M(v)$ and $M(w)$ be indecomposable string modules. 
Denote the corresponding arcs by $\gamma_v$ and $\gamma_w$. Then 
\[
\dim\Ext_\Lambda^4(M(v),M(w))=\left\{
\begin{array}{cl}
2 & \mbox{if $\gamma_v,\gamma_w$  are in Case 1}\\
1 & \mbox{if $\gamma_v,\gamma_w$  are in Case 2}\\ 
0 & \mbox{otherwise} 
\end{array}
\right.
\]
Case 1: at both ends, in the orientation of the surface, $\gamma_w$ 
follows $\gamma_v$ in an internal triangle, see Figure~\ref{fig:w-follows-v}. \\
Case 2: $\gamma_w$ 
follows $\gamma_v$ only at one of the two ends of the arcs. 

In case 1, the two arrow extensions given by $\alpha$ and $\alpha'$ as in 
Figure~\ref{fig:w-follows-v}
induce a basis 
for $\Ext_\Lambda^4(M(v),M(w))$. In case 2, the basis is induced by only  one of them. 

Furthermore, $\Ext^{4+3k}(M(v),M(w))\cong \Ext^4(M(v), M(w))$ for all $k\ge 0$.
\end{theorem}

\begin{figure}[H]
\[
\includegraphics[scale=.8]{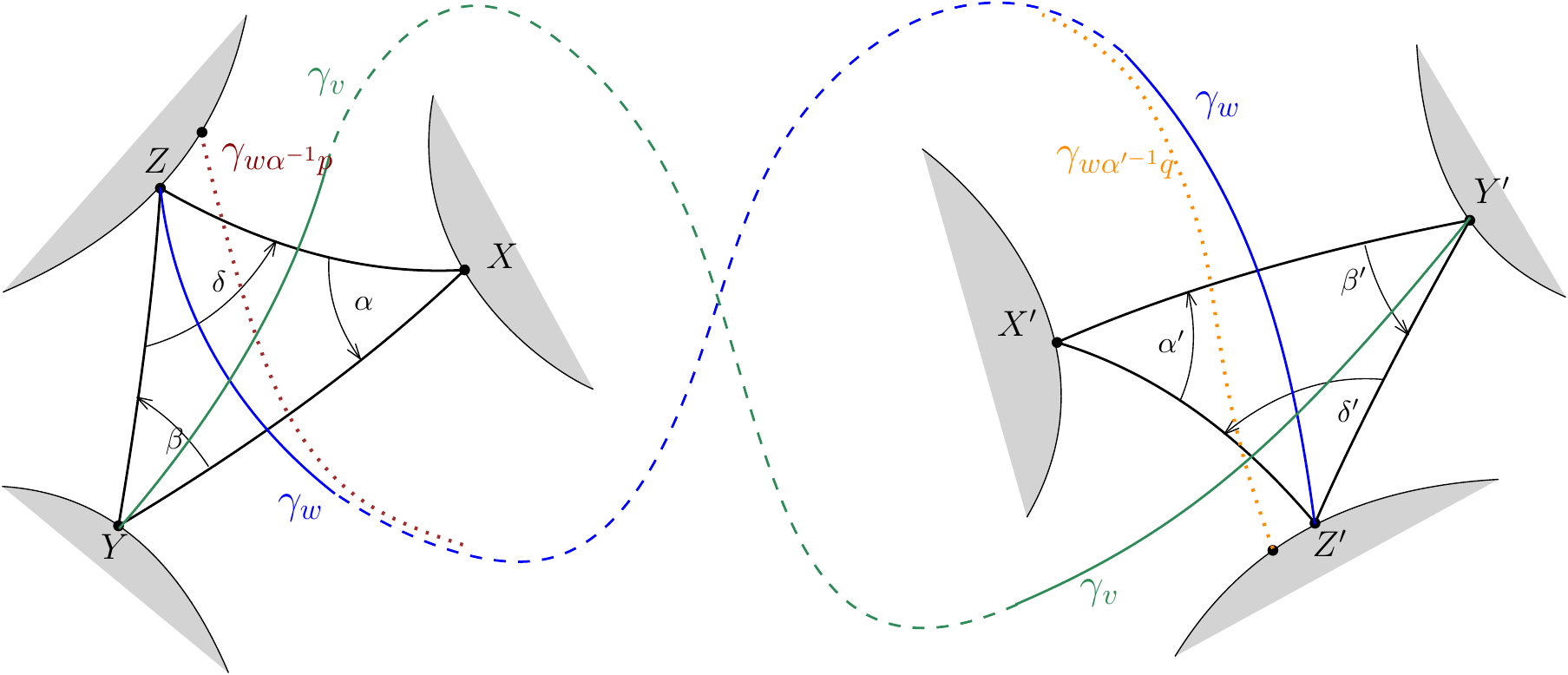}
\]
\caption{Local endpoint configurations of $\gamma_v$ and $\gamma_w$ 
contributing to $\Ext_\Lambda^4(M(v),M(w))$.} 
\label{fig:w-follows-v}
\end{figure}

\begin{proof}
The claim follows with a similar argument as in the proofs of 
Theorem~\ref{thm:dim-ext2} and \ref{thm:ext4}, using  
\[
\begin{array}{lcl} 
\Ext_\Lambda^4(M(v),M(w)) 
& = & \Ext_\Lambda^1(\Omega^2(L_0),M(w)) \oplus 
    \Ext_\Lambda^1(\Omega^2(L_k),M(w)) \\ 
 \\ 
\end{array}
\]
\end{proof}

\begin{rem}
We can describe a basis of $\Ext_\Lambda^4(M(v),M(w))$ in terms of arcs. \\
The basis element corresponding 
to $\Ext_\Lambda^1(\Omega^2(L_0,M(w)))$ is given by the string $w\alpha^{-1}p$ 
where $p$ is the string of $R(\delta)$. 
Its arc $\gamma_{w\alpha^{-1}p}$ starts at the anti clockwise neighbour of 
$Z$ and follows $\gamma_w$ to end at $Z'$. See left side of Figure~\ref{fig:w-follows-v}. \\
The basis element corresponding to $\Ext_\Lambda^1(\Omega^2(L_k),M(w))$ 
is given by the 
string $w\alpha'^{-1} q$ where 
$q$ is the string of $R(\delta')$. 
Its arc $\gamma_{w\alpha'^{-1} q}$ starts at the anti clockwise neighbour of 
$Z'$ and follows $\gamma_w$ to end at $Z$. See right side of Figure~\ref{fig:w-follows-v}. 
\end{rem}

We end this section on gentle algebras arising from triangulations with an example. 
\begin{ex}
Consider the following triangulation of an annulus, given by the black arcs labelled $1,2,\dots, 14$ 
in the figure. 

\begin{figure}[H]
\[
\includegraphics[width=9cm]{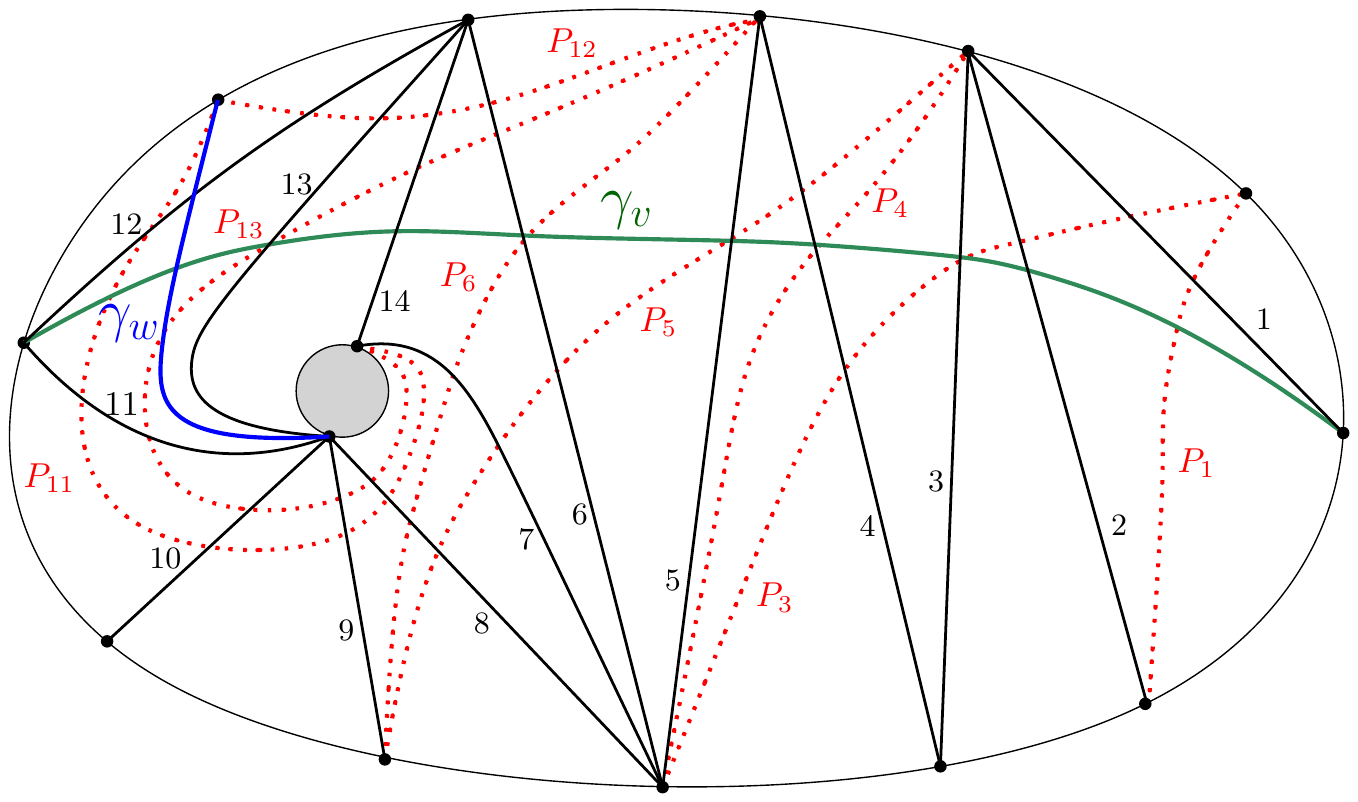}
\]
\caption{A triangulation of an annulus.}\label{fig:triangulation}
\end{figure}

Its algebra is given by the quiver $Q=Q(T)$ 
\[
\includegraphics[width=8cm]{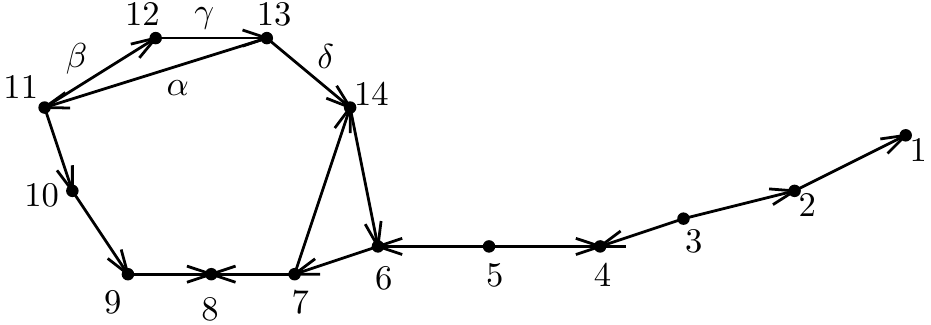}
\]
where the composition of any two arrows in a triangle is zero. 
We compute the projective resolutions for the indecomposable string modules $M=M(v)$ 
and $N=M(w)$ 
corresponding to the two arcs $\gamma_v$ and $\gamma_w$. We use this to determine 
the higher extensions between and give the corresponding arcs. 
Take the indecomposable modules 
$
M=$
\begin{tiny}
$
\xymatrix@C=-.04em@R=-.04em{
13 \\
& 14 & & 5 & &3 \\
 && 6 && 4 && 2 
}
$
\end{tiny} and 
$N=12$ 
corresponding to the arcs $\gamma_w$ and $\gamma_v$ in the figure. 
Note that $\gamma_w$ and $\gamma_v$ both have one 
endpoint in an internal triangle and one endpoint in a triangle 
with one boundary edge.
We first compute the projective resolution of $M$. 
The arcs corresponding to the projective indecomposables 
appearing in the resolution are drawn in red (dotted) in Figure~\ref{fig:triangulation}. They form a zig-zag path 
in the surface. 
\[
\xymatrix@C=.1em@R=.4em{
\cdots \ar[rr]\ar[rdd] &&  P_{11}\ar[rr]\ar[rdd] && P_{13}\ar[rr]\ar[rdd] && P_{12} \ar[rr]\ar[rdd]  && 
  P_{11}\oplus P_6\oplus P_4\oplus P_1\ar[rr]\ar[rdd] && P_{13}\oplus P_5\oplus P_3\ar[rr] && M \\ 
  & & & && && && \\
 & R(\beta)\ar[ruu]&& R(\alpha) \ar[ruu] && R(\gamma)\ar[ruu] && R(\beta)\ar[ruu] && \Omega^1(M)\ar[ruu] 
}
\]

We have $\Omega^1(M) = L_0 \oplus L_1 \oplus L_2 \oplus L_3$ where  
$L_0=R(\alpha)$, $L_1 = P_6, L_2 = P_4, L_3 = P_1$. 
Furthermore, $\Omega^2(M) = \Omega^1(L_0)= R(\beta)=12$ 
and $\Omega^3(M) = \Omega^2(L_0) 
= R(\gamma)$, 
cf.~Lemma~\ref{lm:R-alpha} and Proposition~\ref{prop:syzygy}. 
The projective resolution of $N$ is 
\[
\xymatrix@C=.1em@R=.4em{
\cdots\ar[rdd]\ar[rr]  \ar[rr] &&  P_{13}\ar[rr]\ar[rdd] && P_{12}\ar[rr]\ar[rdd]
 && P_{11}\ar[rr]\ar[rdd]  &&  P_{13}\ar[rr]\ar[rdd] &&  P_{12}\ar[rr] && N \\ 
  & & & && && && \\
 &  R(\gamma) \ar[ruu] && R(\gamma) \ar[ruu] && R(\beta)\ar[ruu] &&R(\alpha)\ar[ruu]
  && \Omega^1(N)\ar[ruu] 
}
\]
where $\Omega^1(N)=R(\gamma)$.

Using this, we determine the higher extensions between $M$ and $N$. It is straightforward to see that 
$\Ext^1(M,N)=\Ext^1(N,M)=0$ and we note that the arcs $\gamma_v$ and $\gamma_w$ do not intersect. 
We also have 
$\Ext^2(M,N)=\Ext^1(\Omega^1(M),N)=0$ 
and $\Ext^2(N,M)=\Ext^1(\Omega^1(N),M)=0$. The second extensions of $N$ by $M$ 
is non-zero: 
$\Ext^2(M,N)=\Ext^1(\Omega^2(M),N)=\Ext^2(R(\alpha),12)\cong K$ where the non-trivial middle term of the 
short exact sequence is 
$M=$
\begin{tiny}
$
\xymatrix@C=-.04em@R=-.04em{
& 11 \\
12 && 10 \\  & & & 9 \\  && & &8
}
$
\end{tiny}. 
$\Ext^3(N,M)=\Ext^3(M,N)=0$ and $\Ext^4(N,M)\cong K$ with middle term of the 
non-trivial extension given by 
$\begin{tiny}
\xymatrix@C=-1em@R=-.04em{12 \\
&M}\end{tiny}
$, while $\Ext^4(M,N)=0$. 
This yields 
\[
\begin{array}{lcl}
\dim\Ext^i(N,M)_{i>0} & =  & (0,0,0,1,0,0,1,0,\dots) 
\\
 & \\
\dim\Ext^i(M,N)_{i>0} & = & (0,1,0,0,1,0,\dots).
\end{array}
\] 
To obtain the projective indecomposables appearing in degree 2 and above in the projective resolution of 
a module ending in an internal triangle, such as $M$, 
we only need to rotate the arcs forming this internal triangle to obtain the arcs 
corresponding to the indecomposables $P(L_0)$, $P(\Omega^1(L_0))$ and $P(\Omega^2(L_0))$ 
where $L$ is a direct summand of the first syzygy of $M$. 
\[
\includegraphics[width=4.8cm]{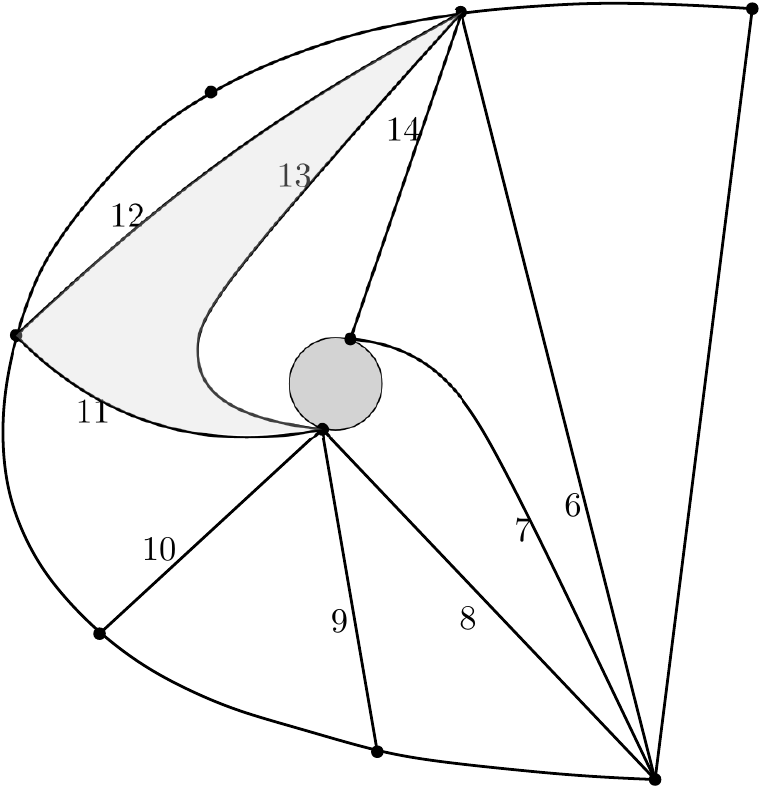}
\hskip1.5cm
\includegraphics[width=4.8cm]{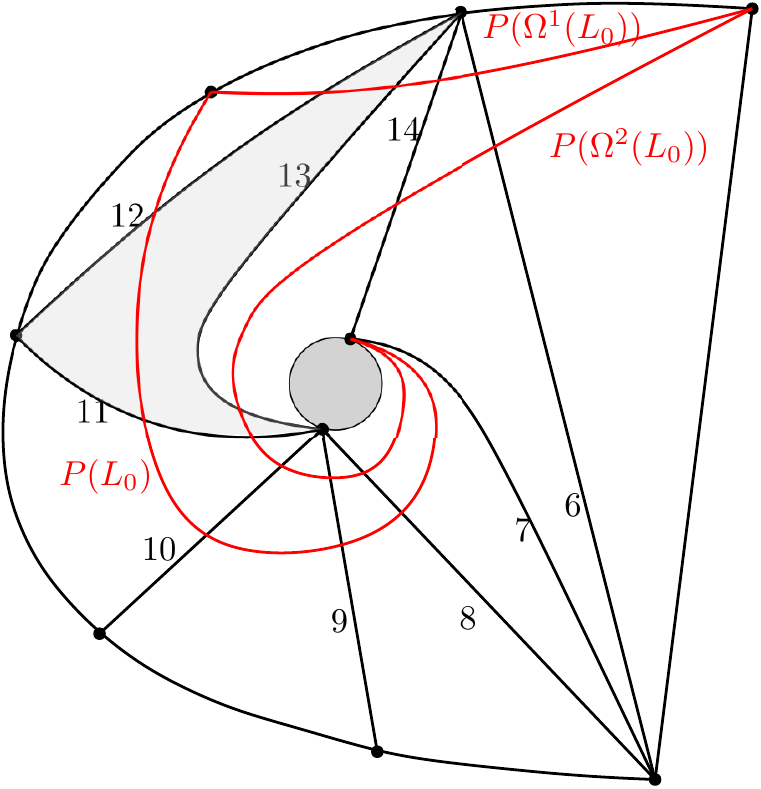}
\]
In other words, 
the projective covers of $L_0$, $\Omega^1(L_0)$ and of $\Omega^2(L_0)$ 
form a triangle. 

\end{ex}

%
\subsection{Geometric interpretation of higher extensions for general gentle algebras}

Similarly as in the surface triangulation case, using the results of Section~\ref{sec:syzygies} and 
of the geometric description of the module category of a gentle algebra in~\cite{BCS2018}, 
a geometric description of higher extensions between 
indecomposable string modules for general gentle algebras can be given as follows.

\begin{rem}\label{rem:geometric-way}
Let $\Lambda=KQ/I$ be gentle, let $M=M(v)$ and $N=M(w)$ be indecomposable string modules 
with associated arcs $\gamma_v$ and $\gamma_w$. Let $P_0$ and $P_k$ be the tiles in which the 
arc $\gamma_v$ starts and ends and $P'_0$ and $P'_{k'}$ be the tiles in which $\gamma_w$ starts and ends.
Then we have: \\
(1) There exists $n_0>0$ such that $\Ext^i(M,N)=0$ for all $i\ge n_0$ if and only if the arcs $\gamma_v$ and 
$\gamma_w$ start or end in no common internal tile. 

\noindent
(2) There exist infinitely many $i>0$ with $\Ext^i(M,N)\ne 0$ if $\{P_0,P_k\}\cap \{P_0',P_{k'}'\}\neq \emptyset$ 
and at least one of the common tiles is internal. 
In this case, the sequence $(\dim \Ext^i(M,N))_i$ 
eventually becomes periodic with period equal to the size of this common tile or of the least common multiple of 
the two common tiles. 
\end{rem}

To finish this section, we consider two examples of gentle algebras, one with infinite and one with finite 
global dimension. 

\begin{ex}
Let $\Lambda$ be the gentle algebra given by the quiver below, with a 4-cycle full of relations. 
By~\cite{BCS2018} (see also \cite{CSP}), $\Lambda$ can be given by a dissection of a polygon as shown 
on the right. 
We consider (higher) extensions between the two indecomposable string modules 
$N=M(w)=\begin{tiny}\xymatrix@C=-.04em@R=-.04em{2\\ 3}\end{tiny}$  
and $M=M(v)=\begin{tiny}\xymatrix@C=-.04em@R=-.04em{4 \\  5}\end{tiny}$. 
\[
\includegraphics[height=3.5cm]{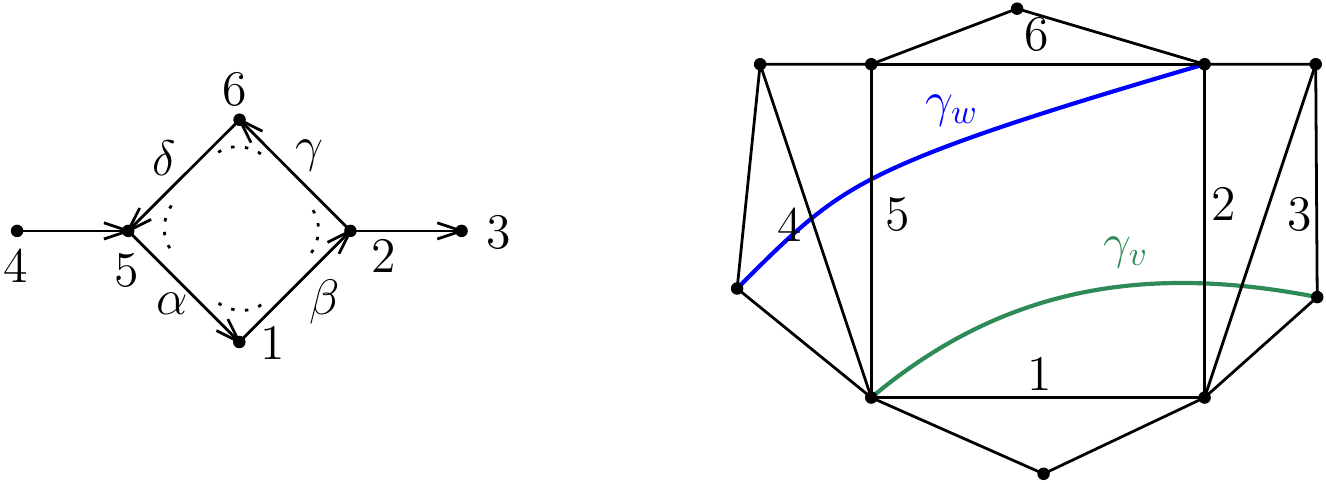}
\]
The projective resolution of $N$ is 
\[
\xymatrix@C=.1em@R=.4em{
\cdots \ar[rr]\ar[rdd]&& P_6\ar[rr]\ar[rdd] &&  P_{2}\ar[rr]\ar[rdd] && P_1\ar[rr]\ar[rdd] && P_5 \ar[rr]\ar[rdd]  && 
  P_6\ar[rr]\ar[rdd] && P_2\ar[rr] && N \\ 
  & & & && && && \\
 & R(\delta)\ar[ruu] && R(\gamma)\ar[ruu]&& R(\beta) \ar[ruu] && R(\alpha)\ar[ruu] 
  && R(\delta)\ar[ruu] && \Omega(N) \ar[ruu] 
}
\]
with $\Omega(N)=R(\gamma)=6$. 
The one of $M$ is: 
\[
\xymatrix@C=.1em@R=.4em{
\cdots \ar[rr]\ar[rdd]&& P_1\ar[rr]\ar[rdd] &&  P_{5}\ar[rr]\ar[rdd] && P_6\ar[rr]\ar[rdd] && P_2 \ar[rr]\ar[rdd]  && 
  P_1\ar[rr]\ar[rdd] && P_4\ar[rr] && M \\ 
  & & & && && && \\
 & R(\beta)\ar[ruu] && R(\alpha)\ar[ruu]&& R(\delta) \ar[ruu] && R(\gamma)\ar[ruu] 
  && R(\beta)\ar[ruu] && \Omega(M) \ar[ruu] 
}
\]
where $\Omega(M)$ is the simple at 1, $R(\beta)=$
\begin{tiny}
$\xymatrix@C=-.04em@R=-.04em{
2 \\
3}
$
\end{tiny}, 
$R(\gamma)=6$, $R(\delta)=5$, $R(\alpha)=1$. 
One obtains
$\Ext^1(N,M)=0$, 
$\Ext^2(N,M)=\Ext^1(\Omega(N),M)\cong K$, $\Ext^3(N,M)=\Ext^1(R(\delta),M)=0$, 
$\Ext^4(N,M)=\Ext^1(R(\alpha),M)=0$, $\Ext^5(N,M)=\Ext^1(R(\beta),M)=0$, so for 
the dimensions, we have 
$\dim\Ext^i(N,M)_{i>0}=(0,1,0,0,0,1,0,0,0,\dots)$. 
\[
\includegraphics[width=4.5cm]{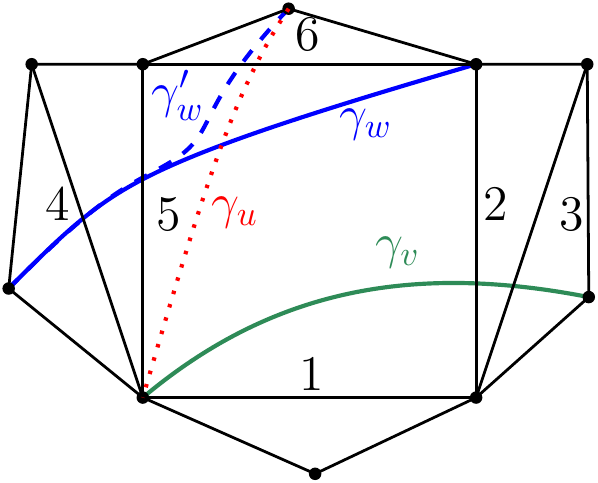}
\hskip 1cm
\includegraphics[width=4.5cm]{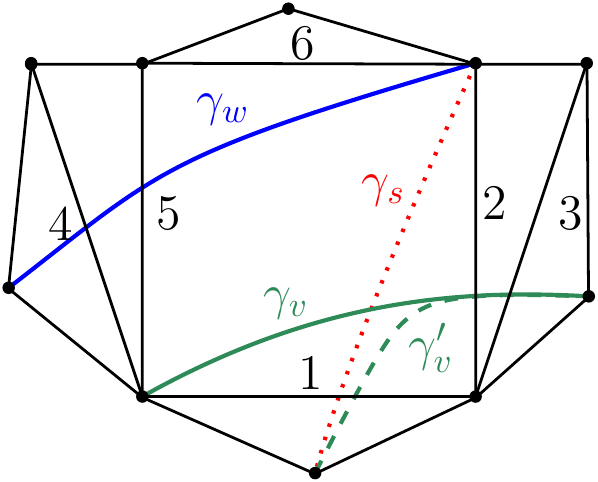}
\]
In the picture on the left, $\gamma_u$ is drawn dotted red, it is the arc corresponding to $\Omega^1(N)$ 
and the arc 
yielding a basis element for $\Ext^2(N,M)$ is $\gamma_w'$. 
Similarly, $\Ext^1(M,N)=0$ $\Ext^2(M,N)=\Ext^1(1,N)\cong K$ $\Ext^3(M,N)=\Ext^4(M,N)=$ $\Ext^5(M,N)=0$. 
In the picture on the right, the red dotted arc $\gamma_s$ corresponds to $\Omega^1(M)$, the 
arc yielding a basis element for $\Ext^2(M,N)$ is $\gamma_v'$. 
\end{ex}

\begin{ex}
Let $\Lambda$ be the gentle algebra given by the quiver below, where the relations are generated 
by $\alpha\beta$. This algebra arises from the polygon dissection on the right. 
\[
\includegraphics[height=3.2cm]{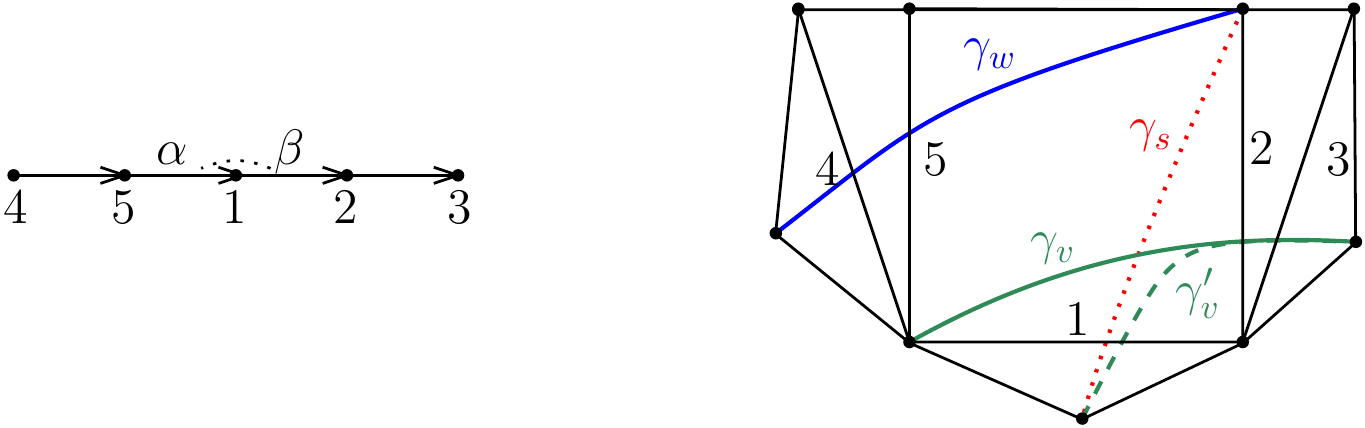}
\]
$N=M(w)=\begin{tiny}\xymatrix@C=-.04em@R=-.04em{2\\ 3}\end{tiny}$  
and $M=M(v)=\begin{tiny}\xymatrix@C=-.04em@R=-.04em{4 \\  5}\end{tiny}$. 
Here, $N$ is projective, so $\Ext^i(N,M)=0$ for all $i>0$. 
The module $M$ is not projective, but has finite projective dimension,  
\[
\xymatrix@C=.1em@R=.4em{
 \ar[rr] 0&&  P_2 \ar[rr]\ar[rdd]  && 
  P_1\ar[rr]\ar[rdd] && P_4\ar[rr] && M \\ 
  &   && && && \\
 &      
  && R(\beta)\ar[ruu] && \Omega(M) \ar[ruu] 
}
\]
with $\Omega(M)=1$ and $R(\beta)=N$. From this, we get 
$\Ext^2(M,N)=\Ext^1(\Omega(M),N)\cong K$ and $\Ext^i(M,N)=0$ for all $i>0$, $i\ne 2$. 
Again, the red dotted arc $\gamma_s$ corresponds to $\Omega^1(M)$ and the 
arc yielding a basis element for $\Ext^2(M,N)$ is $\gamma_v'$.

\end{ex}

\bibliographystyle{acm}
\bibliography{biblio}

\end{document}